\documentclass[11pt]{article}
\usepackage{amsmath} 

\input epsf%
\newtheorem{theorem}{Theorem}[section]
\newtheorem{corollary}{Corollary}[section]

\newtheorem{lemma}{Lemma}[section]

\newtheorem{Definition}{Definition}[section]
\def\qed{{\hfill{\vrule height4pt width3pt depth2pt}}}

\newtheorem{Proof}{Proof.}

\newenvironment{proof}{\begin{Proof} \rm}
    {\hfill {\vrule height4pt width3pt depth2pt}\end{Proof}}

\def\ad#1{\begin{aligned}#1\end{aligned}}  \def\b#1{\mathbf{#1}}
\def\a#1{\begin{align*}#1\end{align*}} \def\an#1{\begin{align}#1\end{align}} 
\def\e#1{\begin{equation}#1\end{equation}} \def\d{\operatorname{div}}
\def\p#1{\begin{pmatrix}#1\end{pmatrix}} 
 \def\vc{\operatorname{\bf curl}} \numberwithin{equation}{section}
 
\def\Qe{$Q_{k+1,k}$-$Q_{k,k+1}$}

\def\boxit#1{\vbox{\hrule height1pt \hbox{\vrule width1pt\kern1pt
     #1\kern1pt\vrule width1pt}\hrule height1pt }}
 \def\lab#1{\boxit{\small #1}\label{#1}}   
  \def\mref#1{\boxit{\small #1}\ref{#1}}    
 \def\meqref#1{\boxit{\small #1}\eqref{#1}}    
 
  \def\lab#1{\label{#1}} \def\mref#1{\ref{#1}} \def\meqref#1{\eqref{#1}}

\begin{document} 

\title {
Superconvergence of the $Q_{k+1,k}$-$Q_{k,k+1}$ divergence-free finite element}

\author {Yunqing Huang and Shangyou Zhang
    } \date{}

\maketitle

\begin{abstract}
   By the standard theory, the stable 
	\Qe/$Q_{k}^{dc'}$ 
       divergence-free element converges with the optimal
       order of approximation
 	for the  Stokes equations,
     but only order $k$ for the velocity in $H^1$-norm 
     and   the pressure  in $L^2$-norm.
   This is due to one polynomial degree less in $y$ direction for
     the first component of velocity, a $Q_{k+1,k}$ polynomial.
   In this manuscript,  we will show a superconvergence of the
     divergence free element that the order of convergence is truly
     $k+1$, for both velocity and pressure.
  Numerical tests are provided confirming the sharpness of the theory.
 
  \vskip 15pt
 

 \vskip 15pt


\end{abstract}

\section{Introduction}

The divergence-free finite element method is mainly for 
  solving incompressible flow problems,
  such as Stokes or Navier-Stokes equations, where
  the finite element space for the pressure is exactly the
    divergence of the finite element space for the velocity.
In such a method, the finite element velocity is divergence-free
   pointwise, i.e.
  the incompressibility condition is enforced strongly.
Traditional finite elements enforce the incompressibility
  weakly, cf. \cite{Raviart,BrezziD}.
That is, in order to satisfy the inf-sup stability condition,
   the incompressibility condition is weakened by either
   enlarging the velocity space or
   decreasing the pressure space.
This often leads to some sub-optimal methods, or a waste of
    computation, due to the imperfect matching of two spaces.
It may lead to inaccurate mass conservation, which is critical
   in certain computational problems.

A fundamental study on the divergence-free element method was done by
   Scott and Vogelius (\cite{Scott-Vogelius,Scott-V})
  that the $P_{k+1}$/$P_k^{dc}$ method
      is stable and consequently of the optimal
   order on 2D triangular grids, for  $k\ge 3$.
Here the velocity space is the continuous
   piecewise-polynomials of degree $(k+1)$ or less while
   the  the pressure space is the discontinuous 
   piecewise-polynomials of degree $k$ or less,
   or the divergence of the velocity, to be precise.
There are several other such divergence-free finite
   elements, cf. \cite{Arnold-Qin,Huang-Zhang,
       Qin,Qin-Zhang,Zhang-3D,Zhang-PS,Zhang-Q}.

Starting from the most popular element, the $Q_1$/$P_0$ element
	(\cite{Boland,Boland-Nicolaides}),
   there is a   series of work on $Q_k$ mixed finite elements on
   rectangular grids in 2D and 3D.
Brezzi and Falk showed that the $Q_{k+1}$/$Q_{k}^{dc}$ element is unstable
	in \cite{Brezzi-Falk}, for any $k\ge 0$.
Here $Q_{k}^{dc}$ denotes the space of discontinuous piecewise-polynomials.
In \cite{Stenberg-Suri}, Stenberg and Suri showed 
    the stability, but a sub-optimal
   order of approximation, for the $Q_{k+1}$/$Q_{k-1}^{dc}$ 
   element for all $k\ge 1$ in 2D.
Bernardi and Maday proved the stability and the optimal
   order of convergence for the  $Q_{k+1}$/$P_{k}^{dc}$ 
   element, cf. \cite{Bernardi-Maday}.
Ainsworth and Coggins established \cite{Ainsworth-Coggins} the 
    stability and the optimal
   order of convergence for the Taylor-Hood $Q_{k+1}$/$Q_{k}$ 
	element, where the pressure space is continuous too.
The Bernardi-Raugel element (\cite{Bernardi-Raugel})
    optimizes the $Q_{k+1}$/$Q_{k-1}^{dc}$ 
   element, when $k=1$,
    by reducing the velocity space to $Q_{1,2}$-$Q_{2,1}$
    polynomials.
Here the first component of velocity in the Bernardi-Raugel element
   is a polynomial of degree $1$ in $x$ direction, but of
    degree 2 in $y$ direction.
To be precise, the Bernardi-Raugel element enrich 
   the $Q_1$ velocity space by face-bubble functions.
Similar to the Bernardi-Raugel element, a
  divergence-free finite element, \Qe/$Q_{k}^{dc'}$ ($k\ge 2$),
     was proposed in 
   \cite{Zhang-Q}, which further optimizes the  Bernardi-Raugel element
	by increasing the polynomial degree of pressure 
        from $(k-1)$ to $k$.
The nodal degrees of freedom of this  divergence-free 
    element and the Bernardi-Raugel element are plotted in 
   Figure \mref{dof}.
This divergence-free 
    element was extended to its lowest-order form, $k=1$, i.e.,
    $Q_{2,1}$-$Q_{1,2}$/$Q_1^{dc'}$, in \cite{Huang-Zhang}.
Here the space $Q_{k}^{dc'}$ for the
   pressure is the space of discontinuous $Q_k$ polynomials with
   all spurious modes removed, i.e.,
   eliminating one degree of freedom at each vertex.
In the construction,  the pressure space is exactly the divergence 
  of the velocity.
Thus, the resulting finite element is divergence-free pointwise.
In such a case, the discrete pressure space can be omitted in the computation.
By an iterated penalty method,  we obtain the  pressure solution
   as a byproduct, cf. \cite{Zhang-PS} and
   Section \mref{s-numerical} below. 
However, by the standard finite element theory developed in 
    \cite{Huang-Zhang,Zhang-Q},  this divergence-free element converges
   at order $k$ only, due to a degree $k$ polynomial in $y$  
   for the first component of $\b u_h$.
This cannot be improved by the standard theory, where the optimal order
   of convergences is derived from the inf-sup stability.
In this manuscript,  we further study this \Qe divergence-free element
  and show its superconvergence, that it does converge at order $k+1$.
Further the velocity of the \Qe divergence-free element may
   be ultraconvergent, i.e., two orders higher than the standard
   convergence, provided the interpolation polynomial is divergence-free.
The extension of this divergence-free element to 3D is straightforward,
   so is its superconvergence property.
 
\begin{figure}[htb] \begin{center} \setlength{\unitlength}{0.6pt}
    \begin{picture}(400,200)(0,0) 
  \def\bo{\put(100,100){\line(0,-1){100}} \put(0,0){\line(1,0){100}}
   \put(100,100){\line(-1,0){100}} \put(0,0){\line(0,1){100}} }
  \def\ux{\begin{picture}(100,100)(0,0)\bo
 \multiput(0,0)(50,0){3}{\circle*{5}} \multiput(0,100)(50,0){3}{\circle*{5}} 
   \put(40,40){$Q_{2,1}$} \end{picture}}
    \def\uy{\begin{picture}(100,100)(0,0)\bo
 \multiput(0,0)(0,50){3}{\circle*{5}} \multiput(100,0)(0,50){3}{\circle*{5}} 
 	\put(40,40){$Q_{1,2}$}   \end{picture} }
 \put(-20,40){$\b u_h$:} \put(20,100){\ux} \put(140,100){\uy}
 \put(-20,150){$\b u_h$:} \put(20,-20){\uy} \put(140,-20){\ux} 
   \put(300,100){ \begin{picture}(100,100)(0,0)\bo  
 \multiput(10,10)(0,80){2}{\circle{6}}  \multiput(90,10)(0,80){2}{\circle{6}}
 	\put(-30,50){$p_h$:} \put(40,40){$Q_{1}$}  \end{picture} }
   \put(300,-20){ \begin{picture}(100,100)(0,0)\bo  \put(50,50){\circle{6}} 
 	\put(-30,50){$p_h$:} \put(12,40){$Q_{0}$}  \end{picture} }
 \end{picture}\end{center}
\caption{Nodes of $\b u_h$/$p_h$ for divergence-free (top)
      and Bernardi-Raugel elements. }   
\lab{dof}
\end{figure}
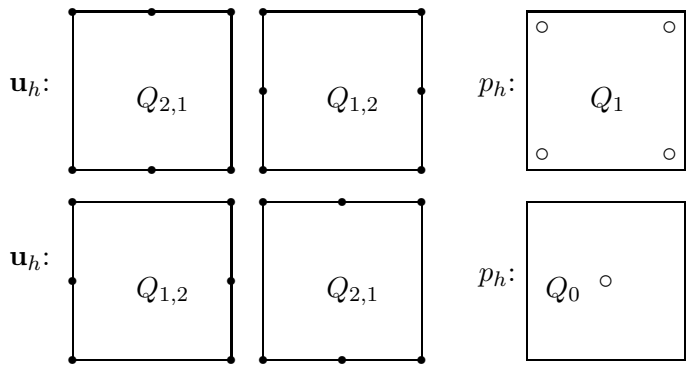 

The rest of the paper is organized as follows.
In Section 2, we define the finite element for the Stationary
   Stokes equations.
In Section 3, we establish a superconvergence for the divergence-free
   element.
In Section \mref{s-numerical},
   we provide some test results confirming the analysis.
In particular, we show the order of convergence of the divergence-free
   element is one higher than that of  the rotated
   Bernardi-Raugel element.

\section{The \Qe \ divergence-free element}

In this section, we shall define the divergence-free  finite element  for 
   the stationary Stokes equations on rectangular grids.
The resulting finite element solutions
   for the velocity are divergence-free point wise.

\def\grad{\nabla}
We consider a model 
     stationary Stokes problem:  Find the velocity 
  $\b u $  and the pressure $p$ on a 2D 
  polygonal domain $\Omega$, which can be subdivided into
    rectangles,  such that
  \e{\lab{e-2} \ad{ - \Delta \b u  + \grad p
            &=\b f \qquad && \hbox{in } \Omega, \\
                 \d \b u  &= 0 && \hbox{in } \Omega, \\
                 \b u  &= \b 0  && \hbox{on } \partial\Omega.  } }
The weak form for \meqref{e-2}
   is:  Find $\b u\in H_{0}^1(\Omega)^2$ and 
    $p\in L_{0}^2(\Omega):=L^2(\Omega)/C=\{ p\in L^2 \mid 
          \int_\Omega p = 0 \}$ such that 
\e{ \ad{  a(\b u,\b v)+b(\b v,p) &=(\b f,\b v)
             &&   \forall \b v \in H^1_0(\Omega)^2 , \\
       b(\b u,q)  &=0 &&  \forall q \in L^2_0(\Omega). }
   \lab{e-v} }
 Here $H^1_{0}(\Omega)^2$ is the subspace of
    the Sobolev space $H^1(\Omega)^2$ (cf. 
      \cite{Ciarlet})  with zero boundary trace, and
   \a{ a(\b u,\b v) &= \int_\Omega \grad \b u \cdot \grad \b v \; dx, \\
       b(\b v, p) &= -\int_\Omega \d \b v \; p \; dx, \\
        (\b f,\b v) &= \int_\Omega  \b f  \; \b v \; dx.}

   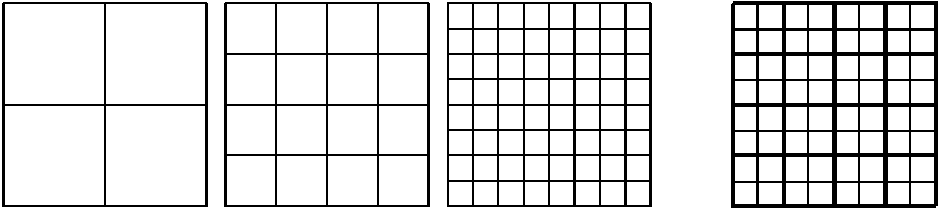
\begin{figure}[htb]\begin{center}\setlength\unitlength{1.2pt}
    \begin{picture}(300,60)(70,10)
 \put(70,0){\begin{picture}(70,64)(0,0)
           \multiput(0,0)(32,0){3}{\line(0,1){64}}
          \multiput(0,0)(0,32){3}{\line(1,0){64}}\end{picture}}
 \put(140,0){\begin{picture}(70,64)(0,0)
           \multiput(0,0)(16,0){5}{\line(0,1){64}}
          \multiput(0,0)(0,16){5}{\line(1,0){64}}\end{picture}}
 \put(210,0){\begin{picture}(70,64)(0,0)
           \multiput(0,0)(8,0){9}{\line(0,1){64}}
          \multiput(0,0)(0,8){9}{\line(1,0){64}}\end{picture}}

  \put(300.3,.3){\begin{picture}(70,64)(0,0)
           \multiput(0,0)(16,0){5}{\line(0,1){64}}
          \multiput(0,0)(0,16){5}{\line(1,0){64}}\end{picture}}
  \put(299.7,-.3){\begin{picture}(70,64)(0,0)
           \multiput(0,0)(8,0){9}{\line(0,1){64}}
          \multiput(0,0)(0,8){9}{\line(1,0){64}}\end{picture}}

 \end{picture}  
\end{center}
\caption{Three levels of grids, and a macro-element
    grid (for $k=1$ only).} 
\lab{f-grid}
\end{figure}

The finite element grids are defined by, cf. Figure \mref{f-grid},
   \a{ {\cal T}_h &=\left\{ K \mid \cup K =\overline\Omega, \
	 K=[x_a,x_b]\times[y_c, y_d]
   \  \hbox{ with size }  \
      h_K=\max\{x_b-x_a, y_d-y_c\}\le h \right\}. }
We further assume, only for the lowest-order element $k=1$ in \meqref{e-space},
   that the rectangles in grid ${\cal T}_h$ 
   can be combined into groups of four to form a macro-element
   grid:
  \a{ {\cal M}_h &=\left\{ M \mid M=\cup_{i=1}^4 K_i 
	=[x_{i-1},x_{i+1}]\times [y_{j-1}, y_{j+1}], \
	K_i \in {\cal T}_h, \ \cup K_i =\Omega \right\}. }
See the 4th diagram in Figure \mref{f-grid}.
The polynomial spaces are  defined by
  \a{ Q_{k,l}=\{ \sum_{i\le k, j\le l} c_{ij} x^i y^j\}, \qquad
     Q_{k} = Q_{k,k}. }
The \Qe ($k\ge 1$)   element spaces are 
  \an{\lab{e-space} 
    \b V_h &= \left\{ \b v_h \in C(\Omega)^2  \mid 
         \left. \b v_h \right|_{K}\in Q_{k+1,k}\times Q_{k,k+1}
       \ \forall K\in {\cal T}_h,
    \hbox{ and } \ \left.\b u_h \right|_{\partial\Omega} =0 \right\}, \\
 \lab{e-pressure}
      P_h &= \left\{ \d \b u_h \mid \b u_h \in \b V_h \right\}. }
Since $\int_\Omega p_h = \int_\Omega \d \b u_h =\int_{\partial\Omega} \b u_h
    =0$
   for any $p_h\in P_h$,  we conclude that
     \a{ \b V_h \subset   H_0^1(\Omega)^2,\quad
        P_h \subset L_0^2(\Omega), }
 i.e., the mixed-finite element pair is conforming.
The resulting system of finite element equations for \meqref{e-v} is:
   Find $\b u_h\in \b V_h$ and  $p_h\in P_h$ such that 
\e{ \ad{  a(\b u_h,\b v)+b(\b v,p_h) &=(\b f,\b v)
             &&   \forall \b v \in \b V_h, \\
       b(\b u_h,q)  &=0 &&  \forall q \in P_h. }
   \lab{e-finite} }

Traditional mixed-finite elements require the inf-sup condition
   to guarantee the existence of discrete solutions.
As \meqref{e-pressure} provides a compatibility between the
   discrete velocity and the discrete pressure spaces,  
the linear system of equations \meqref{e-finite} always has a
   unique solution, cf. \cite{Zhang-PS}. 
Furthermore, such a solution $\b u_h$ is divergence-free:
by the second equation in \meqref{e-finite} and the definition of $P_h$ in
    \meqref{e-pressure},  
   \e{\lab{e-u1} b(\b u_h, q) =
         b(\b u_h, -\d \b u_h) = \| \d \b u_h \|_{L^2(\Omega)^2}^2
    =0. } 
In this case, i.e., $\b V_h\subset \b Z:=\{\d\b v \ | \
	\b v  \in H^1_0(\Omega)^2 \} $,
   we call the mixed finite element a divergence-free element.
It is apparent that the discrete velocity solution is divergence-free
   if and only if the discrete pressure finite element space is the
   divergence of the discrete velocity finite element space, i.e.,
  \meqref{e-pressure}. 

We note that by \meqref{e-pressure}, $P_h$ is a subspace of discontinuous,
   piecewise bilinear polynomials.  
As  singular vertices are present
  (see \cite{Scott-Vogelius,Scott-V,Huang-Zhang,Zhang-Q}),
    $P_h$ is a proper subset
  of the discontinuous piecewise $Q_1$ polynomials.
It is possible, but difficult to find a local basis for $P_h$.
But on the other side, it is the special interest of the divergence-free
   finite element method
   that the space $P_h$ can be omitted in computation and the discrete
   solutions approximating the pressure function in the Stokes 
   equations can be obtained as byproducts, 
   if an iterated penalty method is adopted to solve
   the system \meqref{e-finite}, 
  cf. \cite{book-Fortin,BrezziD,Brenner-Scott,Scott-penalty,Zhang-PS} 
   for more information.

\section{Superconvergence }

As usual, the superconvergence is obtained by the method of
   integration by parts, cf. \cite{Chen-Huang,Yan}.
But we have a long series of lemmas dealing with each term
   in the bilinear forms $a(\cdot,\cdot)$ and $b(\cdot,\cdot).$

For a convenience in referring components of the vector
  velocity,  we define the two inhomogeneous polynomial spaces:
  \an{\lab{V1}
	 V_{h,1} &=\{ \phi\in H^1_0(\Omega) \mid
	 \phi|_K \in Q_{k+1,k} \ \forall K\in {\cal T}_h \}, \\
	\lab{V2}
	 V_{h,2} &=\{ \phi\in H^1_0(\Omega) \mid
	 \phi|_K \in Q_{k,k+1} \ \forall K\in {\cal T}_h \},}
$k\ge 1$.  That is, 
   \a{ \b V_h = V_{h,1} \times V_{h,2}, \quad k\ge 1. }

\def\squ#1{\begin{picture}(100,100)(0,0)
    \multiput(0,0)(100,0){2}{\multiput(0,0)(0,100){2}{\circle*{4}}}
  \put(100,100){\line(0,-1){100}} \put(0,0){\line(1,0){100}}
   \put(100,100){\line(-1,0){100}} \put(0,0){\line(0,1){100}}#1\end{picture}}
\begin{figure}[htb] \begin{center} \setlength{\unitlength}{0.8pt}
    \begin{picture}(400,110)(-30,0) 
   \def\di{\hskip-2pt\vrule height2pt width4pt depth2pt}
 \put(-50,0){\squ{\put(40,-12){$Q_{3,2}$}
   \multiput(0,0)(100,0){2}{\multiput(0,0)(0,100){2}{\circle*{2}}}
   \multiput(0,0)(100,0){2}{\multiput(0,0)(0,100){2}{\circle{7}}}
  \multiput(0,33.33)(100,0){2}{\multiput(0,0)(0,33.33){2}{\circle{5}}}
    \multiput(25,0)(25,0){3}{\multiput(0,0)(0,100){2}{\circle*{4}}}
    \multiput(25,33.33)(25,0){3}{\multiput(0,0)(0,33.33){2}{\di}}}}
  \put(100,0){ \squ{ \put(40,-12){$Q_{2,3}$}
   \multiput(0,0)(100,0){2}{\multiput(0,0)(0,100){2}{\circle{7}}}
    \multiput(33.33,0)(33.33,0){2}{\multiput(0,0)(0,100){2}{\circle{5}}}
    \multiput(0,25)(0,25){3}{\multiput(0,0)(100,0){2}{\circle*{4}}}
    \multiput(31.33,23)(33.33,0){2}{\multiput(0,0)(0,25){3}{$\diamond$}}}}
  \put(240,80){\circle*4}\put(240,80){\circle7\quad$u_I(a_i^K)=u(a_i^K)$}
  \put(240,60){\circle*{4}}\put(240,60){\quad\ $\int u_I p_2ds=\int u p_2ds$}
  \put(240,40){\circle5\quad$\int u_I p_1ds=\int u p_1ds$}
   \put(240,20){\di\quad$\int_K u_I q_{2,1}d\b x=\int_K u q_{2,1}d\b x$}
   \put(238,-2){$\diamond$}
    \put(240,0){\quad$\int_K u_I q_{1,2}d\b x=\int_K u q_{1,2}d\b x$}
 \end{picture}\end{center}
\caption{ Three types of interpolation nodes,
    $k=3$. }   
\lab{nodes}
\end{figure}
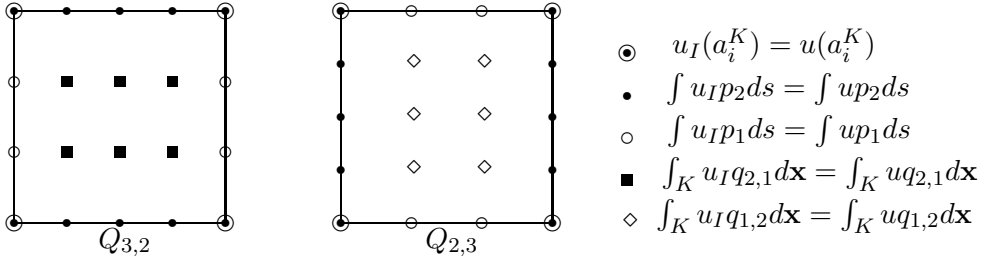

The interpolation operator $\b I_h$ is defined for the two components of
   $\b u$:
 \an{ \nonumber \b I_h : H^1_0(\Omega)\times  H^1_0(\Omega) \to
         V_{h,1} \times V_{h,2}, \\
      \b I_h \b u =  \b I_h \p{u \\ v} = \p{ u_I \\ v_I }.
	\lab{interpolation} }
 To define $u_I$ at the Lagrange nodes,  
   we define its vertex nodal values, then internal
   edge values, and finally internal values, by solving the
   following equations sequentially (see Figure \mref{nodes}):
    \an{\lab{int-n}  (u-u_I)(a_i^K)&=0 &&\hbox{at four
	 vertices of $K, \ \forall K\in {\cal T}_h$}, \\ 
	\lab{int-k-1}\int_{y=y_j} (u-u_I) p_{k-1}(x) dx & =
	    0 &&
	\hbox{ on the top and bottom edges of $K$},\\
	\lab{int-k-2}\int_{x=x_i} (u-u_I) p_{k-2}(y) dy & =
	   0 &&
	\hbox{ on the left and right edges of $K$. }\\
 	\lab{int-1-2}\int_{K}(u- u_I) q_{k-1,k-2}d\b x & =
	    0 &&
	\hbox{ on the square  $K$},  }
  where $p_k\in P_k$, the space of 1D polynomials of degree $k$ or less, 
    and $q_{k,l}\in Q_{k,l}$.
 By rotating $x$ and $y$,  $v_I$ is defined similarly/symmetrically to $u_I$.

\begin{lemma} (two-order superconvergence) \lab{l-sup1}
   For any $Q_{k+1,k}$ function $\psi\in V_{h,1}$, defined in \meqref{V1},
	  for any $u\in H^{k+3}(\Omega)$, and for all $k>1$, 
     \e{\lab{sup-xx}
      |\int_\Omega (u-u_I)_x \psi_x d\b x|
	\le  Ch^{k+2}\|u\|_{H^{k+3}} \|\psi\|_{H^1}. }
\end{lemma}

 \begin{proof}    
  We first consider the estimation on the reference element
   $\hat K=[-1,1]^2$.
  Since $\psi\in Q_{k+1,k}$, we have an exact Taylor expansion:
 \an{ \lab{taylor}
	 \psi_x (x, y) = \psi_x(x, 0) + y \psi_{xy}(x,0)
	+\cdots+ \frac{y^{k-1}}{(k-1)!}\psi_{xy^{k-1}}(x,0)
	+ \frac{y^k}{k!}\psi_{xy^k}(x,0),}
  where $\psi_x(x, 0) $ and all $\psi_{xy^j}(x,0)$ are $P_k$ polynomials
    in $x$ only.
    We will perform the integration by parts repeatedly. First, 
	for the lower order terms in \meqref{taylor}, we 
   notice that, by the definition of $u_I$ in 
	\meqref{int-k-1} and  \meqref{int-1-2},
  \an{ \nonumber &\int_{\hat K} (u-u_I)_x y^j\psi_{xy^j}(x,0) d\b x 
    \\ = & \int_{-1}^1 (u-u_I)y^j\psi_{xy^j}(x,0)|_{x=-1}^{x=1}dy
	 \nonumber-\int_{\hat K} (u-u_I)y^j\psi_{x^2y^j}(x,0) d\b x \\
	 \lab{k-2-terms} 
       = & 0 \qquad \hbox{ when } j=0, 1, \dots, k-2. }
  Please be aware that $\psi_{x^2y^j}(x,0)\in P_{k-1}(x)$.
 Hence, we need to deal with only the last two terms in \meqref{taylor}.

For the last two terms in \meqref{taylor}, in order to do integration by parts, 
   we express   polynomials $y^{k-1}$ and $y^k$ by derivatives
   of another polynomial. 
   \an{  \lab{s-k} s_k(y) &=\frac {(y^2-1)^{k+1}}{(2k+2)!}  
	    = \frac { y^{2k+2} }{(2k+2)!} 
	    - \frac {(k+1) y^{2k} }{(2k+2)!} + \cdots = 
	   \frac { y^{2k+2} }{(2k+2)!} + \tilde p_{2k}(y) , \\
	\lab{s-j-b}
	  s_k^{(j)}(\pm1)&= 0, \quad j=0,1,\cdots,k, \\
	\lab{s-yk}  s_k^{(k+2)}(y) &= \frac 1{k!} y^k +  p_{k-2}(y),  
    }
 Here $\tilde p_{2k}(y)$ and $ p_{k-2}(y)$ denote a polynomial
   of degree $2k$ and $(k-2)$, respectively.
 We note that, as in \meqref{k-2-terms}, the integral of $(u-u_I)_x$ against
    $p_{k-2}(y)$ is zero.
 Thus, surprisingly simple, we have
 \an{ \nonumber
     &\quad \int_{\hat K} (u-u_I)_x \psi_x(x,y) dx dy
   \\ &= \int_{\hat K} (u-u_I)_x (s_{k-1}^{(k+1)}(y)\psi_{xy^{k-1}}(x, 0) 
	\nonumber		+ s_{k }^{(k+2)}(y) \psi_{xy^k}(x,0))dxdy\\
    & =   \int_{-1}^1 \left[(u-u_I)_x (s_{k-1}^{(k)}(y) \psi_{xy^{k-1}}(x, 0)
   \nonumber + s_{k }^{(k+1)}(y) \psi_{xy^{k}}(x,0) )
		\right]_{y=-1}^{y=1}dx    \\
    &\quad - \int_{\hat K}(u-u_I)_{xy} (s_{k-1}^{(k)}(y)  \psi_{xy^{k-1}}(x, 0)
      +  s_{k }^{(k+1)}(y)   \psi_{xy^k}(x,0))dxdy.
   \lab{ux-1}  
   } 
Let us consider the first boundary integral in \meqref{ux-1},
    on the top edge of the square $\hat K$.
  By \meqref{int-n} and \meqref{int-k-1},  
  \an{  & \nonumber
      \int_{-1}^1 (u-u_I)_x(x,1)s_{k-1}^{(k)}(1) \psi_{xy^{k-1}}(x, 0) dx \\
	=&  \left[ (u-u_I)(x,1) s_{k-1}^{(k)}(1)
		 \psi_{xy^{k-1}}(x, 0)\right]_{x=-1}^1
	\nonumber \\
	&  \quad   - s_{k-1}^{(k)}(1) \int_{-1}^1 (u-u_I) (x,1) 
		 \psi_{x^2y^{k-1}}(x, 0) dx  
	=   0 , \lab{b-0}   }
  noting again that $\psi_{x^2y^{k-1}}(x,0)$ is a $P_{k-1}$ polynomial 
    in $x$ only.
The other   boundary integral in \meqref{ux-1} is also $0$ as
	$\psi_{x^2y^{k}}(x,0)\in P_{k-1}$ too:
  \a{  & \quad
      \int_{-1}^1 (u-u_I)_x(x,1)s_{k}^{(k+1)}(1) \psi_{xy^{k}}(x, 0) dx \\
	 & =  \left[ (u-u_I)(x,1) s_{k}^{(k+1)}(1)
		 \psi_{xy^{k}}(x, 0)\right]_{x=-1}^1
	\nonumber \\
	&  \quad   - s_{k}^{(k+1)}(1) \int_{-1}^1 (u-u_I) (x,1) 
		 \psi_{x^2y^{k}}(x, 0) dx  
	=   0 .   }
That is the boundary integrals in \meqref{ux-1} are all zero.
We repeat the integration by parts in this direction, while 
   the boundary terms would be zero by \meqref{s-j-b} and \meqref{int-k-1}.
By the integration by parts $k$ times more,
  \meqref{ux-1}    would be
  \an{    &\quad \nonumber \int_{\hat K} (u-u_I)_x \psi_x dx dy
    \\  &  \nonumber 
         =   - \int_{\hat K}(u-u_I)_{xy} (s_{k-1}^{(k)}  \psi_{xy^{k-1}}(x, 0)
      +  s_{k }^{(k+1)}   \psi_{xy^k}(x,0))dxdy
  \\  &  \nonumber 
         =  \int_{\hat K}(u-u_I)_{xy^2} (s_{k-1}^{(k-1)}  \psi_{xy^{k-1}}(x, 0)
      +  s_{k }^{(k)}   \psi_{xy^k}(x,0))dxdy
    \\ &  \nonumber 
        \quad -  \int_{-1}^1 (u-u_I)_{xy}(x,1)
	(s_{k-1}^{(k-1)}  \psi_{xy^{k-1}}(x, 0)
      +  s_{k }^{(k)}   \psi_{xy^k}(x,0))_{y=-1}^{y=1}dx
  \\  &  \nonumber 
       = \int_{\hat K}(u-u_I)_{xy^2} (s_{k-1}^{(k-1)}  \psi_{xy^{k-1}}(x, 0)
      +  s_{k }^{(k)}   \psi_{xy^k}(x,0))dxdy
  \\ &  
       =  (-1)^{k+1} \int_{\hat K}(u-u_I)_{xy^{k+1}} 
	(s_{k-1}  \psi_{xy^{k-1}}(x, 0)
      +  s_{k }'  \psi_{xy^k}(x,0))dxdy. \lab{ux-last} }
We will perform the integration by parts one last time.  
But this time, we will treat the two terms in the last integral
   differently.
\a{ \int_{\hat K}(u-u_I)_{xy^{k+1}} 
	 s_{k-1}  \psi_{xy^{k-1}}(x, 0) dxdy 
   &= -\int_{\hat K}(u-u_I)_{x^2y^{k+1}} 
	 s_{k-1}  \psi_{y^{k-1}}(x, 0) dxdy \\
   &\quad +\int_{-1}^1 \left[ (u-u_I)_{xy^{k+1}} s_{k-1}  \psi_{y^{k-1}}(x, 0)
	\right]_{x=-1}^{x=1} dy, \\
  \int_{\hat K}(u-u_I)_{xy^{k+1}} 
	 s_{k}'  \psi_{xy^{k}}(x, 0) dxdy 
   &= -\int_{\hat K}(u-u_I)_{xy^{k+2}} 
	 s_{k}  \psi_{xy^{k}}(x, 0) dxdy. }
For the second integral, the boundary term disappears by the condition
   \meqref{s-j-b}.
For the first integral, we note that the boundary integrals will
   be cancelled due to the opposite line integrals on two
   sides of the vertical edge $x=x_i$ or due to 
     the boundary condition on $\psi$.
We note also that the $(k+1)$st and $(k+2)$nd
    partial derivatives on $u_I$ above are all zero.
Hence, we get \meqref{sup-xx} by summing over the estimation on
  all rectangles $K\in {\cal T}_h$, plus a scaling and the
  Schwartz inequality, 
  \a{  & \left|\int_\Omega (u-u_I)_x \psi_x d\b x\right|
     \\=&\left|\sum_{K} \int_K (u-u_I)_x \psi_x d\b x \right|
	=\left|\sum_{K} \int_{\hat K} (u-u_I)_x \psi_x d\b x \right| \\
  = &\left|\sum_{K} (-1)^{k+2} \int_{\hat K}\left( u_{x^2y^{k+1}} 
	 s_{k-1}  \psi_{y^{k-1}} 
	+  u_{xy^{k+2}}  s_{k}  \psi_{xy^{k}} \right) d\b x \right| \\
    \le &\sum_K  C |u |_{H^{k+3}(\hat K)}  |\psi  |_{H^1(\hat K)} 
      = C \sum_K h^{k+2}  |u |_{H^{k+3}( K)}    |\psi |_{H^1( K)}
    \\ \le & C h^{k+2} |u |_{H^{k+3}( \Omega)} |\psi  |_{H^1(\Omega)}. }
 We note that the semi $H^1$-norm is needed above to bound
    $\psi_{y^{k-1}}$. 
 Thus $k>1$ is required.
  \end{proof}

 In the proof,  we can see that the decrease of one degree polynomial in 
   $y$ does not change the super-approximation of $Q_{k+1,k}$ in $x$
	direction.
 After \meqref{ux-last}, if we skip the last step of integration by parts,
  we would get the following corollary. 
  That is, we avoid $|\psi_{y^{k-1}}|_{L^2}$ when $k=1$ which
  cannot be bounded by $|\psi|_{H^1}$.
 
\begin{corollary} (one-order superconvergence)  \lab{c-sup1} 
      For any $Q_{k+1,k}$ function $\psi\in V_{h,1}$, defined in \meqref{V1},
	  for any $u\in H^{k+2}(\Omega)$, and for all $k\ge 1$,
     \e{\lab{sup-xx-1}
      |\int_\Omega (u-u_I)_x \psi_x d\b x|
	\le  Ch^{k+1}\|u\|_{H^{k+2}} \|\psi\|_{H^1}. }
    \qed
\end{corollary}

Symmetrically, switching $x$ and $y$ in Lemma \mref{l-sup1},
   we prove the following lemma.

\begin{lemma} (two-order superconvergence) \lab{l-sup2} 
      For any $Q_{k,k+1}$ function $\psi\in V_{h,2}$, defined in \meqref{V2},
	and for any $u\in H^{k+3}(\Omega)$, if $k>1$,
     \e{\lab{sup-yy}
      |\int_\Omega (u-u_I)_y \psi_y d\b x|
	\le  Ch^{k+2}\|u\|_{H^{k+3}} \|\psi\|_{H^1}. }
    \qed
\end{lemma}

For the same reasons in Corollary \mref{c-sup1},  we get the following
   corollary from Lemma \mref{l-sup2}.

\begin{corollary} (one-order superconvergence) \lab{c-sup2} 
      For any $Q_{k,k+1}$ function $\psi\in V_{h,2}$, defined in \meqref{V2},
    for any $u\in H^{k+2}(\Omega)$, and for all $k\ge 1$,
     \e{\lab{sup-yy-1}
      |\int_\Omega (u-u_I)_y \psi_y d\b x|
	\le  Ch^{k+1}\|u\|_{H^{k+2}} \|\psi\|_{H^1}. }
    \qed
\end{corollary}

Though the interpolation order is $(k+2)$ in above two lemmas,
   only the $(k+1)$ order in two corollaries can be achieved in
   computation due to the coupling of terms in mixed formulation.
We prove the approximation properties in the lower polynomial direction
   next.
Now, even for $k=1$, we have a two-order superconvergence.

\begin{lemma} (two-order superconvergence)   \lab{l-sup3} 
   For any $Q_{k+1,k}$ function $\psi\in V_{h,1}$, defined in \meqref{V1},
	  for any $u\in H^{k+3}(\Omega)$, and for all $k\ge 1$,
     \e{\lab{sup-xx-y}
      |\int_\Omega (u-u_I)_y \psi_y d\b x|
	\le  Ch^{k+2} \|u\|_{H^{k+3} } \|\psi\|_{H^1}. }
\end{lemma}

 \begin{proof}    
  Again, we first consider the estimation on the reference element
   $\hat K=[-1,1]^2$.
  Since the polynomial degree in $y$ is too low, we do
    Taylor expansion in $x$ direction, different from the last lemma.
  \a{ \psi_y (x, y) =  
    \psi_y(0,y)+x\psi_{xy}(0,y)+\cdots +
	\frac {x^k}{k!}\psi_{x^ky}(0,y)+
	\frac {x^{k+1}}{(k+1)!}\psi_{x^{k+1}y}(0,y).}
  Again, similar to \meqref{taylor}, the integral of $(u-u_I)_y$ against
     $x^j$ terms are zero if $j\le k-1$,
 \a{ &\quad \int_{\hat K} (u-u_I)_y x^j \psi_{x^jy}(0,y) dxdy 
	\\ & =  \int_{-1}^1 \left[(u-u_I)x^j \psi_{x^jy}(0,y)
	  \right]_{y=-1}^{y=1} dx
	  - \int_{\hat K} (u-u_I) x^j \psi_{x^jy^2}(0,y) dxdy 
	=0, }
  noting that $x^j \psi_{x^jy^2}(0,y)\in Q_{k-1,k-2}$.
  Using the polynomial function $s_k(x)$ defined in \meqref{s-k}
   we have, cf. \meqref{ux-1}, 
 \a{ &\int_{\hat K} (u-u_I)_y \psi_y dx dy
    \\=&\int_{\hat K} (u-u_I)_y  ( s_k^{(k+2)}(x) \psi_{x^ky} (0,y)
          + s_{k+1}^{(k+3)}(x) \psi_{x^{k+1}y}(0,y)) dxdy\\
    =& \int_{-1}^1 \left[ (u-u_I)_{y} (s_k^{(k+1)}(x) \psi_{x^ky}(0,y)
          + s_{k+1}^{(k+2)}(x) \psi_{x^{k+1}y}(0,y))\right]_{x=-1}^{x=1}dy\\
	&  - \int_{\hat K} (u-u_I)_{xy} (s_k^{(k+1)}(x) \psi_{x^ky} (0,y)
          + s_{k+1}^{(k+2)}(x) \psi_{x^{k+1}y}(0,y)) dxdy.
  } Here, for the first time integration by parts,
     the boundary integral disappeared by \meqref{int-n},
    $(u-u_I)(\pm 1,\pm 1)=0$.
 In the next $(k+1)$ times of integration by parts, the 
    boundary integrals on $x=\pm1$ would be zero, directly by
	the boundary condition \meqref{s-j-b} of $s_k(x)$.
 \a{ \int_{\hat K} (u-u_I)_y \psi_y dx dy
    &= (-1)^{k+2} \int_{\hat K} (u-u_I)_{x^{k+2}y} 
	  (s_k  \psi_{x^ky} (0,y)
          + s_{k+1}' \psi_{x^{k+1}y}(0,y)) dxdy.
  } 
Thus,
  \a{ \left|  \int_{\hat K} (u-u_I)_y \psi_y dx dy \right|
    & \le   C \|u_{x^{k+2}y}\|_{L^2(\hat K)}  \|\psi_y \|_{L^2(\hat K)} \\
    & \le  C |u|_{H^{k+3}(\hat K)}  |\psi \|_{H^1(\hat K)}. }
 The rest proof repeats that of Lemma \mref{l-sup1}.
  \end{proof}

As for above lemmas and corollaries, we can get the following
   corollary from Lemma \mref{l-sup3}

\begin{corollary} (two-order superconvergence) \lab{c-sup3} 
  For any $Q_{k+1,k}$ function $\psi\in V_{h,1}$, defined in \meqref{V1},
	 for any $u\in H^{k+2}(\Omega)$, and for all $k\ge 1$,
     \e{\lab{sup-xx-y-1}
      |\int_\Omega (u-u_I)_y \psi_y d\b x|
	\le  Ch^{k+1} \|u\|_{H^{k+2} } \|\psi\|_{H^1}. }
    \qed
\end{corollary}

\begin{corollary} \lab{c-sup4} 
  For any $Q_{k,k+1}$ function $\psi\in V_{h,2}$, defined in \meqref{V2},
	and for any $u\in H^{k+3}(\Omega)$,  and for all $k\ge 1$,
     \an{\lab{sup-yy-x}
      |\int_\Omega (u-u_I)_y \psi_y d\b x|
	\le  Ch^{k+2} \|u\|_{H^{k+3} } \|\psi\|_{H^1}, \\
	\lab{sup-yy-x-1}
      |\int_\Omega (u-u_I)_y \psi_y d\b x|
	\le  Ch^{k+1} \|u\|_{H^{k+2} } \|\psi\|_{H^1}. }
    \qed
\end{corollary}

Now we study the superconvergence in the both bilinear forms.

\begin{lemma} \lab{sup-ab} 
   For any $(\b v_h, q_h)\in \b V_h\times P_h$, defined in 
     \meqref{e-space} and \meqref{e-pressure}, and for any
    $\b u\in H^3(\Omega)\cap H^1_0(\Omega)$, 
     \an{\lab{sup-a}  
	|a(\b u-\b I_h \b u, \b v_h) | &\le C h^{k+2}
	   \|\b u\|_{H^{k+3}(\Omega)^2}
	   \|\b v_h\|_{H^1(\Omega)^2}, \quad k>1, \\
   \lab{sup-a-1}  
	|a(\b u-\b I_h \b u, \b v_h) | &\le C h^{k+1}
	   \|\b u\|_{H^{k+2}(\Omega)^2}
	   \|\b v_h\|_{H^1(\Omega)^2},  \quad k\ge 1,\\
      \lab{sup-b} 
	|b(\b u-\b I_h \b u, q_h) | &\le C h^{k+1} \|\b u\|_{H^{k+2}(\Omega)^2}
	   \|q_h\|_{L^2(\Omega)},  \quad k\ge 1,
	   }
    where $\b I_h \b u$ is the interpolation of $\b u$ defined
     by \meqref{interpolation}.
\end{lemma}

 \begin{proof}  
  \meqref{sup-a} is a combination of 
   \meqref{sup-xx}, \meqref{sup-xx-y},  \meqref{sup-yy-x}, 
        and \meqref{sup-yy}.
   \meqref{sup-a-1} is a combination of 
   \meqref{sup-xx-1}, \meqref{sup-xx-y-1},  \meqref{sup-yy-x-1}, 
        and \meqref{sup-yy-1}.

  For
   \meqref{sup-b}, we will lose one order of convergence.
  Let  $q_h=\d\b w_h$ for some $\b w_h=(\phi,\psi)
	\in \b V_h$.
  We have, denoting $\b u=(u,v)$,
  \a{ b(\b u-\b I_h \b u, q_h) &=
	\sum_K \int_K ((u-u_I)_x+(v-v_I)_y)(\phi_x+\psi_y) d\b x. 
    }
 Here we have two old integrals,
     $\int_K (u-u_I)_x \phi_x d\b x$ and $\int_K (v-v_I)_y \psi_y d\b x$,
    and two new integrals,
   $\int_K (u-u_I)_x \psi_y d\b x$ and $\int_K (v-v_I)_y \phi_x d\b x$.
The approximation order can be one order higher for the two old
   integrals.
For the two new integrals, by symmetry, 
   we consider $\int_K (u-u_I)_x \psi_y d\b x$.
 We use the following
    Taylor expansion on the reference element $\hat K$ in the
   $y$ direction.
  We note that the Taylor expansion in $x$ direction would lead to
    a too high order polynomial  in $y$ direction 
      each term in \meqref{taylor3} below.
  \an{ \lab{taylor3}
    \psi_y(x,y)=\psi_y(x,0)+ y\psi_{y^2}(x,0)+\cdots
	+\frac{y^{k-1}}{(k-1)!} \psi_{y^k}(x,0) 
	+\frac{y^{k}}{k!} \psi_{y^{k+1}}(x,0). }
 Here all $\psi_{y^j}(x,0)$ are polynomials of degree $k$ in
	$x$.
  That is, a generic term $y^j \psi_{y^{j+1}}(x,0) \in Q_{k,j}$.
  This is the same as the generic term $y^j \psi_{xy^j}(x,0)$ in
    the early Taylor expansion \meqref{taylor}.
 Thus repeating the proof of  Lemma
    \mref{l-sup1}, we get
  \a{ \int_{\hat K} (u-u_I)_x \psi_y d\b x
    &=\int_{\hat K} (u-u_I)_x( s_{k-1}^{(k+1)}\psi_{y^k}(x,0)
	+s_k^{(k+2)} \psi_{y^{k+1}}(x,0)) dx dy  \\
   &= (-1)^{k+1} \int_{\hat K}  u_{xy^{k+1}}( s_{k-1}\psi_{y^k}(x,0)
	+s_k' \psi_{y^{k+1}}(x,0)) dx dy.
   }
  For the second integral, we can do an integration by parts to raise
   one more order.  But we are limited
    by the first integral above to get only
   \a{ \left|\int_{\hat K} (u-u_I)_x \psi_y d\b x\right|
     \le  \|u\|_{H^{k+2}(\hat K)} \| \psi\|_{H^1(\hat K)}. }
  Similarly, we have the same bound for 
	$\left|\int_{\hat K} (u-u_I)_y \psi_x d\b x\right|$.
 \meqref{sup-b} follows by the Schwartz inequality and the scaling
     of referencing mappings.
 \end{proof}

Finally, we estimate the approximation to $p$.

\begin{lemma}  \lab{l-supp} 
   For any  function $\b v_h\in \b V_{h}$, defined in \meqref{e-space},
	and for any $p\in H^{k+1}(\Omega)\cap L^2_0(\Omega)$,
      \e{\lab{sup-p}
      |\int_\Omega  \d \b v_h (p-p_I) d\b x|
	\le  Ch^{k+1} \|\b v_h\|_{H^1} \| p \|_{H^{k+1}}, }
    where $p_I$ is a special nodal interpolation of $p$ in $P_h$,
   defined in \meqref{pI} below.
\end{lemma}

 \begin{proof}   We note that 
    $P_h$ are discontinuous $Q_k$ functions,
    $P_h=\d \b V_h$.
    We define an interpolation operator for $P_h$ via that $\b I_h$
     for $\b V_h$ defined in \meqref{interpolation}.
    For a $p\in H^2(\Omega)\cap L^2_0(\Omega)$, 
      Arnold, Scott and Vogelius shown in \cite{Arnold-Scott-Vogelius}
     that there is a $\b w\in  H^3(\Omega)^2\cap H^1_0(\Omega)^2$, 
    such that \a{ \d \b w = p, \quad \hbox{ and \ }
	   \|\b w\|_{H^3}\le C\|p\|_{H^2}. }
  For simplicity, we assume the above lifting exists up to order $k+1$.
    We define 
	\e{\lab{pI} p_I = \d \b w_I,}
      for $\b w_I=\b I_h \b w $ defined by \meqref{interpolation}.
    In order to use \meqref{sup-b}, we use notations:
  \a{ \b w=\p{u\\ v}, \ \b w_I =\p{ u_I \\ v_I }, \
        \b v_h =\p{\phi  \\ \psi }. }
  Repeating the proof in Lemma \mref{sup-ab}, we get 
   \a{  |\int_\Omega  \d \b v_h (p-p_I) d\b x|
	& = |\int_\Omega  \d \b v_h \d(\b w-\b w_I) d\b x| \\
       &= | \int_\Omega ((u-u_I)_x+ (v-v_I)_y) 
		(\phi_x+\psi_y )d\b x \\
       &\le C h^{k+1}\left|\p{ u\\ v}\right|_{H^{k+2}}
	   \left |\p{\phi\\\psi}\right |_{H^{1}}
	\\ & \le C h^{k+1} \left\|p\right\|_{H^{k+1}}
                 \left\|\b v_h\right\|_{H^1}. } 
\end{proof}

We derive the main theorem.

\begin{theorem}\lab{t-sup}  The finite element solution 
	$(\b u_h, p_h)$ of \meqref{e-finite} has the
   following superconvergence property, one order higher than the
   optimal order, 
  \an{\lab{truncation2}  
    \|\b u_h - \b I_h \b u \|_{H^1} +  \| p_h - p_I\|_{L^2} & \le C h^{k+1}
     (\|\b u\|_{H^{k+2} } + \|p\|_{H^{k+1}}), }
   where the interpolations $(\b I_h \b u, p_I)$ are
    defined in \meqref{interpolation} and \meqref{pI}.
  \end{theorem}

 \begin{proof}

By the inf-sup condition shown in \cite{Huang-Zhang,Zhang-Q}, 
   it follows that, cf. 
    \cite{Raviart}, for all $(\b w_h, r_h)\in \b V_h\times P_h$,
  \e{\lab{biginfsup}
   \sup_{(\b v_h,q_h)\in \b V_h\times P_h} \frac
   {a(\b w_h,\b v_h) + b(\b v_h, r_h) +
	  b(\b w_h, q_h) } {\| \b v_h\|_{H^1}
	   + \| q_h\|_{L^2} }
    \ge C (\|\b w_h\|_{H^1} +\|r_h\|_{L^2}). }
By Corollary \mref{sup-ab} and Lemma \mref{l-supp}, we have
\a{    &\quad \|\b u_h - \b I_h \b u\|_{H^1}+\|p_h - p_I\|_{L^2} \\
   & \le  C \sup_{(\b v_h, q_h)\in\b V_h\times P_h}
	\frac{ a(\b u_h - \b I_h \b u ,\b v_h)
	 + b(\b v_h, p_h-p_I) + b(\b u_h - \b I_h \b u, q_h) }
      {\|\b v_h\|_{H^1}+\|q_h\|_{L^2} } \\
    & =   C  \sup_{(\b v_h, q_h)\in\b V_h\times P_h}
	\frac{ a(\b u  - \b I_h \b u ,\b v_h)
	 + b(\b v_h, p-p_I) + b(\b u  - \b I_h \b u, q_h) }
      {\|\b v_h\|_{H^1}+\|q_h\|_{L^2} } \\
    & \le  C h^{k+1}( \|\b u\|_{H^{k+2}} + \|p\|_{H^{k+1}}). }
 Note that, due to the pointwise divergence free property, we have
  above that
  \a{ b(\b u_h - \b I_h \b u, q_h) = 
	b(-\b I_h \b u, q_h) = b(\b u-\b I_h \b u,q_h). }
 \end{proof} 

Here, to be precise, 
    we do not have a superconvergence for $p$ in
   Theorem \mref{t-sup}.
As $P_h$ are degree-$k$ polynomials, the  best order 
   approximation to $p$ in $L^2$-norm would be $(k+1)$.
 However, due to the mixed formulation, the
   convergence of $p_h$ to $p$ is limited to the
     optimal order convergence of $\b u_h$, which is $(k-1)$ in $H^1$-norm
  as $\b u_h$ has polynomial degree $k$ only in $y$ direction
    for its first component.
In this sense, the superconvergence result \meqref{truncation2}
   does lift the order of $p_h$ by one.
For $k>1$, we may have two order superconvergence for the
   velocity.
Such numerical examples are shown in \cite{Zhang-Q} and in next section.
That is, for some special functions $\b u$, $\b I_h \b u$ might be
  also in the divergence-free subspace of $\b V_h$.  
If so, we have a two-order superconvergence result.

\begin{theorem} (two-order superconvergence)
   \lab{t-sup2}  For some solution $\b u$ of \meqref{e-2}, if
   \a{ \b I_h \b u\in \b Z_h
		:=\left\{ \b z_h \in \b V_h \mid \d \b z_h = 0 \right\}, }
   where $\b I_h$ is defined in \meqref{interpolation}, and if
  $k>1$,  then 
  \an{\lab{e-sup2}
    \|\b u_h - \b I_h \b u \|_{H^1}  & \le C h^{k+2}
   \|\b u\|_{H^{k+3}}. } 
  \end{theorem}

 \begin{proof} By \meqref{sup-a},
     limited in to the divergence-free subspace, 
  \a{ \|\b u_h - \b I_h \b u\|_{H^1} & \le  \sup_{\b w_h\in \b Z_h}
     \frac{ a(\b u_h - \b I_h \b u, \b w_h) }{ \|\b w_h\|_{H^1} }\\
	  &= \sup_{\b w_h\in \b Z_h}
     \frac{ a( \b u - \b I_h \b u, \b w_h) }{ \|\b w_h\|_{H^1} } 
    \le C h^{k+2} \|\b u\|_{H^{k+3}}. } 
 \end{proof}

\section{Numerical tests}\lab{s-numerical}

In this section, we report some results of numerical experiments on
   the \Qe element
	for the stationary Stokes equations
   \meqref{e-2} on the unit square $\Omega=[0,1]^2$.
The  grids $ {\cal T}_h$ are depicted in Figure \mref{f-grid},
   i.e., each squares are refined into 4 sub-squares each level.
The initial grid, level one grid, is simply the unit square.
 
We choose an exact solution for  the Stokes equations \meqref{e-2}:
   \e{\lab{e-solution} \b u= \vc g,  \quad p=\Delta g. }
Here \a{    g = 2^8 (x^3-x^4)^2  (y^3-y^4)^2. }
So we can compute
   the right hand side function $\b f$ for \meqref{e-2} as
   \e{\lab{e-test} 
       \ad{\b f & = -\Delta \vc g + \nabla \Delta g. } } 
We note that, unlike \cite{Zhang-Q,Huang-Zhang},
  we intentionally choose a non-symmetric solution so that
   no ultraconvergence would happen, which does not exist 
    in general. 
The solution $p$ is plotted in Figure \mref{f-solutions}.

\begin{figure}[htb]\setlength\unitlength{1in}\begin{center}
    \begin{picture}(4.5,1.75)(0,.25) 
   \put(0,0){ \epsfysize=2in \epsfbox{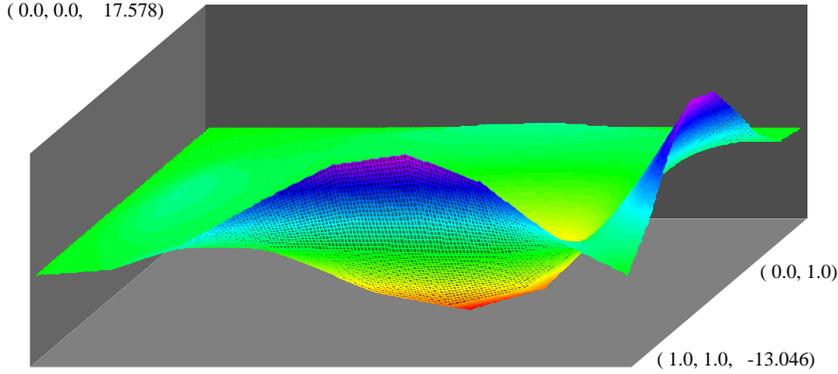}}
    \end{picture}\end{center}
\caption{ The solution $p$ (the errors are shown in Figure \mref{f-errors}.) }   
\lab{f-solutions}
\end{figure} 

We compute the Stokes solution on refined grids, cf. Figure \mref{f-grid},
  by the divergence \Qe element \meqref{e-space} and by the 
   rotated Bernardi-Raugel element \cite{Bernardi-Raugel,BrezziD,Raviart}:
 \an{ \lab{e-BR}
    \b V_h^{BR} &= \left\{ \b v_h \in C(\Omega)^2\cap H^1_0(\Omega)^2  \mid 
         \left. \b v_h \right|_{K}\in Q_{k+1,k}\times Q_{k,k+1}
       \ \forall K\in {\cal T}_h  \right\}, \\
    \nonumber P_h^{BR} &= \left\{ q_h \in L^2_0(\Omega) \mid 
	  q_h |_{K}\in Q_{k-1} \right\}. }
Following the analysis in \cite{Huang-Zhang}, the stability of
   the rotated Bernardi-Raugel element would be proved.
For the rotated Bernardi-Raugel element, the system of finite element equations 
  is solved by the Uzawa iterative method, cf. 
   \cite{BrezziD,Raviart,book-Fortin}.
The stop criterion is the difference $|p_h^{(n)} -p_h^{(n-1)}|\le 10^{-6}$.
We list the number of Uzawa iterations in the data tables by
  \#Uz.
Here the interpolation operators are standard Lagrange nodal
   interpolations \cite{Ciarlet}.

 \begin{table}[htb]
  \caption{ The errors  $\b e_h=\b u-\b I_h \b u$ and
   $\epsilon_h=p-p_I$ for  \meqref{e-solution}.}
\lab{b-1} 
\begin{center}  \begin{tabular}{c|rr|rr|rr|r}  
\hline & $ |\b e_h|_{L^2}$ &$h^n$ &  $ |\b e_h|_{H^1}$ & $h^n$ &
    $ \|\epsilon_h\|_{L^2}$   &$h^n$ & \\ \hline
   & \multispan{5} \hfil \Qe divergence-free element, $k=1$ \hfil &&\#it  \\ \hline
 2& 0.264345&   & 1.341770&   & 5.965379&   &   4\\
 3& 0.102329&1.4& 0.795594&0.8& 1.896372&1.7&   4\\
 4& 0.026839&1.9& 0.219469&1.9& 0.481076&2.0&   3\\
 5& 0.006773&2.0& 0.055901&2.0& 0.120363&2.0&   3\\
 6& 0.001697&2.0& 0.014035&2.0& 0.030083&2.0&   3\\
 7& 0.000424&2.0& 0.003512&2.0& 0.007520&2.0&   3\\
\hline   & \multispan{5} \hfil rotated Bernardi-Raugel element \meqref{e-BR}
         \hfil &&\#Uz  \\ \hline
 2& 0.570990&   & 3.531380&   & 7.497615&   &  29\\
 3& 0.244967&1.2& 3.028368&0.2& 6.943183&0.1&  65\\
 4& 0.074335&1.7& 1.797533&0.8& 3.300598&1.1& 136\\
 5& 0.019849&1.9& 0.946426&0.9& 1.575390&1.1& 297\\
 6& 0.005080&2.0& 0.481087&1.0& 0.762341&1.0& 330\\
 7& 0.001281&2.0& 0.241916&1.0& 0.373990&1.0& 204\\
    \hline 
\end{tabular}\end{center} \end{table}
 
For the \Qe divergence-free element,  the pressure does not enter
   into computation, but is obtained as a byproduct, because
   $P_h=\d \b V_h$.
The resulting linear system of \Qe divergence-free element equations
  can be formulated as symmetric positive definite.
Then the iterated penalty method \cite{book-Fortin,Zhang-Q}
  can be applied to obtain the divergence-free finite element
   solution for the velocity, and a byproduct $p_h=\d \b w_h$
   for the pressure.
In our computation, the iterated penalty parameter is $2000$.
The stop criterion is the divergence $\|\d \b u_h^{(n)}\|_0 \le 10^{-9}$.
 The number of iterated penalty iterations is also listed as
   \#it in
  the data tables. 

In Table \mref{b-1},  we list the errors in various norms for the
   \Qe divergence-free element and for the rotated Bernardi-Raugel element,
   for $k=1$.
It is clear that the order of convergence is 2, one order higher
   than that of latter.
We note that the convergence order is only 2 for 
    $Q_{2,1}$-$Q_{1,2}$ divergence-free elements in $L^2$-norm, i.e.,
   no $L^2$-superconvergence.
But we do see $L^2$-superconvergence for $k>1$ next.

 \begin{table}[htb]
  \caption{ The errors  $\b e_h=\b u-\b I_h \b u$ and
   $\epsilon_h=p-p_I$ for \meqref{e-solution}.}
\lab{b-2} 
\begin{center}  \begin{tabular}{c|rr|rr|rr|r}  
\hline & $ |\b e_h|_{L^2}$ &$h^n$ &  $ |\b e_h|_{H^1}$ & $h^n$ &
    $ \|\epsilon_h\|_{L^2}$   &$h^n$ & \\ \hline
   & \multispan{5} \hfil  
    \Qe divergence-free element, $k=2$ \hfil &&\#it  \\ \hline
 1& 0.322530&0.0& 1.580066&0.0& 3.546405&0.0&   3\\
 2& 0.071851&2.2& 0.699614&1.2& 1.010498&1.8&   4\\
 3& 0.005510&3.7& 0.089611&3.0& 0.131816&2.9&   3\\
 4& 0.000355&4.0& 0.010471&3.1& 0.015587&3.1&   3\\
 5& 0.000022&4.0& 0.001280&3.0& 0.001904&3.0&   3\\
 6& 0.000001&4.0& 0.000159&3.0& 0.000236&3.0&   3\\
 7& 0.000000&4.0& 0.000020&3.0& 0.000029&3.0&   3\\
\hline   & \multispan{5} \hfil  
   rotated  Bernardi-Raugel element \meqref{e-BR} \hfil &&\#Uz  \\ \hline
 1& 0.645475&0.0& 4.250791&0.0& 1.143688&0.0&  27\\
 2& 0.191342&1.8& 2.518701&0.8& 5.499136&---&  67\\
 3& 0.025892&2.9& 0.673622&1.9& 1.621194&1.8& 100\\
 4& 0.003307&3.0& 0.172036&2.0& 0.441596&1.9& 156\\
 5& 0.000419&3.0& 0.043543&2.0& 0.113029&2.0& 266\\
 6& 0.000053&3.0& 0.010954&2.0& 0.028424&2.0& 130\\
 7& 0.000007&3.0& 0.002747&2.0& 0.007117&2.0& 101\\
    \hline 
\end{tabular}\end{center} \end{table}

In Table \mref{b-2},  we list the computation results for $k=2$
   elements.  Again, the divergence-free element is one order 
    higher than the rotated Bernardi-Raugel element.
To show the difference in the two elements,  we plot the
   errors by two elements on level 4 grid in Figure \mref{f-errors}.
One can see the advantage of the divergence-free element, which
   fully utilizes the approximation power of $\b u_h$ by lifting
   the pressure polynomial degree.
Of course, another  advantage is the divergence-free 
  solution after such a lift.
We finally report the results for $k=3$ in Table \mref{b-3}.
All numerical results confirm the theory, and also show the
  sharpness of the superconvergence analysis.

\begin{figure}[htb]\setlength\unitlength{1in}\begin{center}
    \begin{picture}(4,3.75)(0,0.25) 
  \put(0,2){ \epsfysize=2in \epsfbox{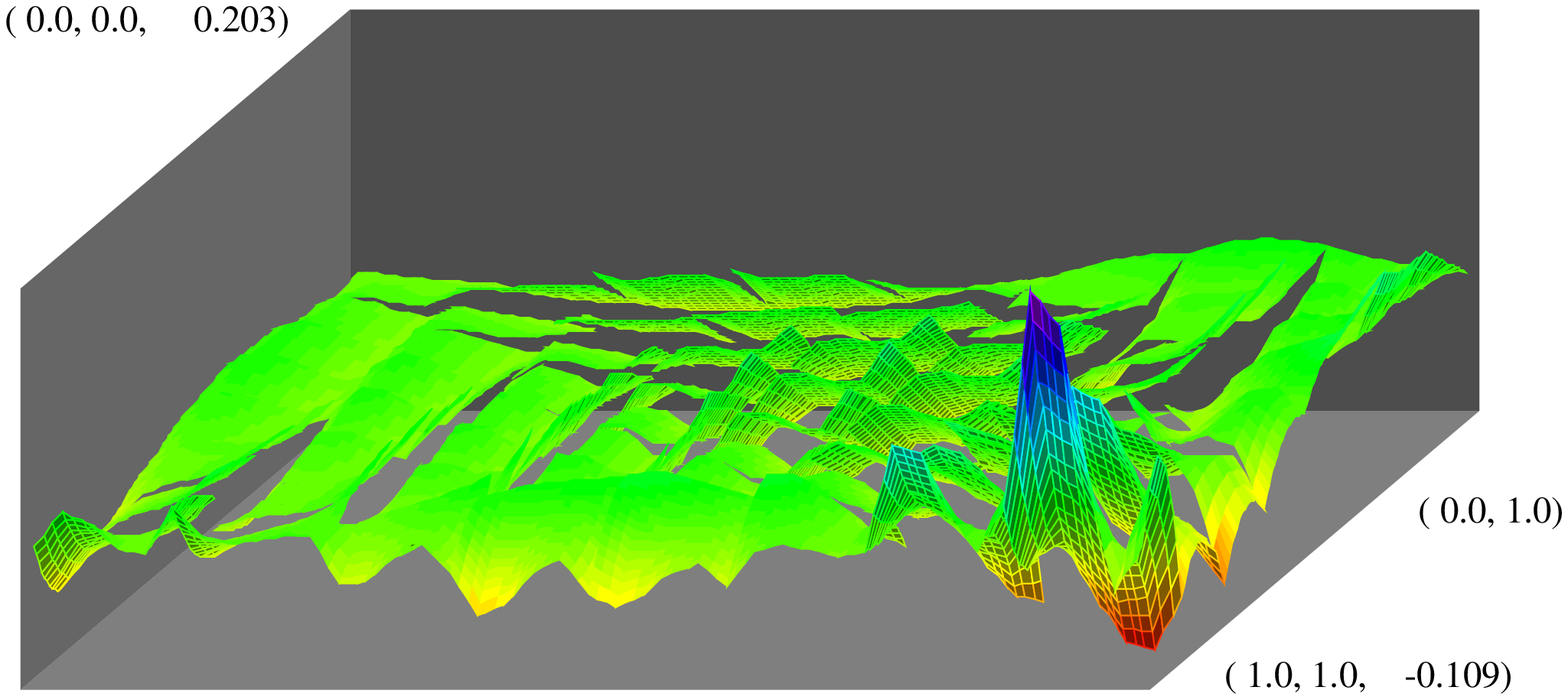}}
   \put(0,0){ \epsfysize=2in \epsfbox{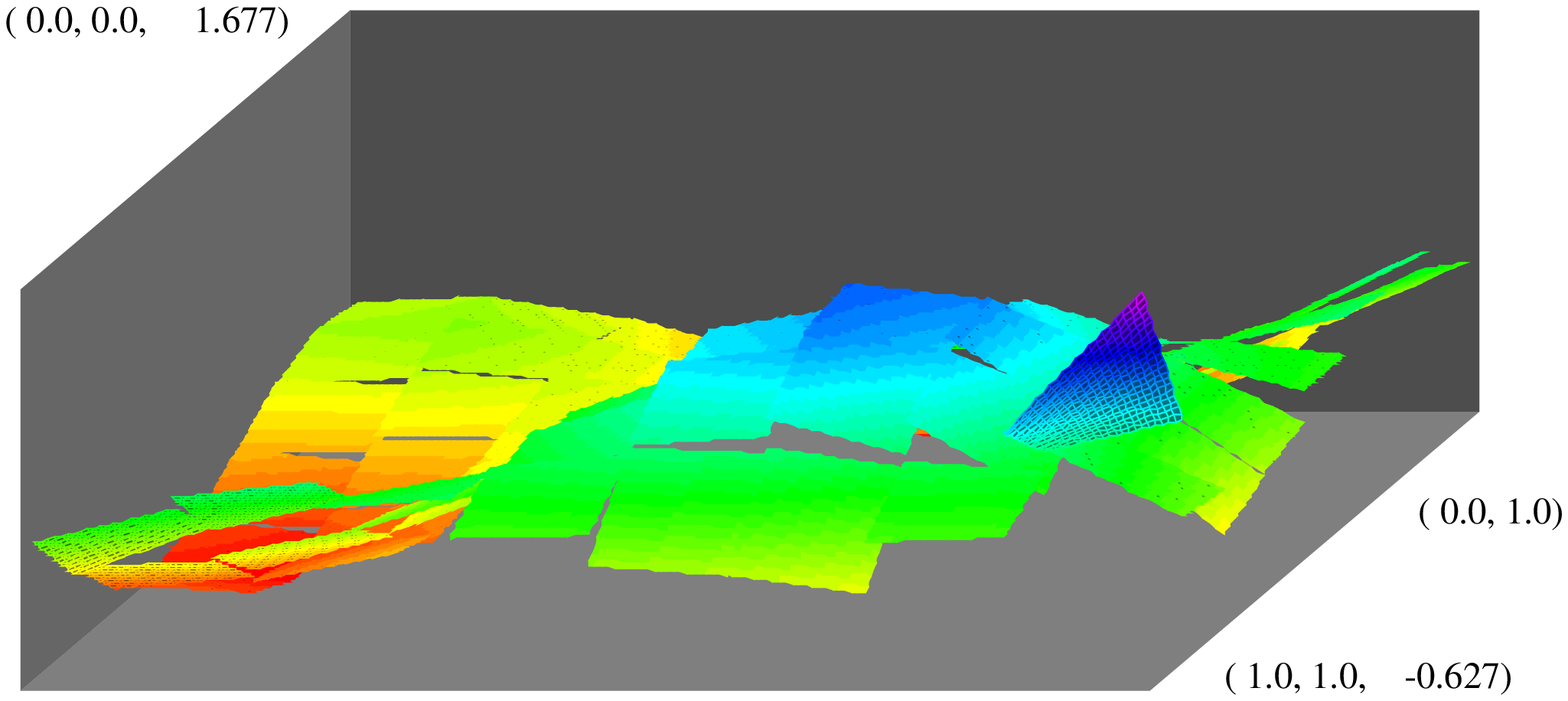}}
  \end{picture}\end{center}
\caption{ The errors of $p_h$ for the divergence-free (top)
   and BR elements. }   
\lab{f-errors}
\end{figure}

\begin{table}[htb]
  \caption{ The errors  $\b e_h=\b u-\b I_h \b u$ and
   $\epsilon_h=p-p_I$ for \meqref{e-solution}.}
\lab{b-3} 
\begin{center}  \begin{tabular}{c|rr|rr|rr|r}  
\hline & $ |\b e_h|_{L^2}$ &$h^n$ &  $ |\b e_h|_{H^1}$ & $h^n$ &
    $ \|\epsilon_h\|_{L^2}$   &$h^n$ & \\ \hline
   & \multispan{5} \hfil \Qe divergence-free element, $k=3$
	 \hfil &&\#it  \\ \hline
  1& 0.123142&0.0& 1.128619&0.0& 1.642992&0.0&   4\\
 2& 0.004515&4.8& 0.065512&4.1& 0.128938&3.7&   3\\
 3& 0.000147&4.9& 0.003911&4.1& 0.008007&4.0&   3\\
 4& 0.000004&5.0& 0.000234&4.1& 0.000494&4.0&   3\\
 5& 0.000000&5.0& 0.000014&4.0& 0.000031&4.0&   3\\
\hline   & \multispan{5} \hfil 
   rotated Bernardi-Raugel element \meqref{e-BR} \hfil &&\#Uz  \\ \hline
1& 0.374364&0.0& 3.512050&0.0& 6.061521&0.0&  57\\
 2& 0.021063&4.2& 0.375407&3.2& 0.736746&3.0&  76\\
 3& 0.001597&3.7& 0.058926&2.7& 0.117000&2.7& 123\\
 4& 0.000111&3.8& 0.008169&2.9& 0.013672&3.1& 177\\
 5& 0.000007&3.9& 0.001065&2.9& 0.001666&3.0& 102\\
    \hline 
\end{tabular}\end{center} \end{table}

Finally, we test the two-order superconvergence in Theorem \mref{t-sup2}.
We choose a symmetric function as the exact solution of the Stokes
	equations \meqref{e-2}:
  \an{\lab{e-solution2} \b u = \vc g, \quad
	g=2^8 (x-x^2)^2(y-y^2)^2. }
Comparing to the data in Table \mref{b-3},  we can see, in Table
	\mref{b-4},  that the
   velocity does converge with another order higher than the
   optimal order.
This is predicted in \meqref{e-sup2}.
Here the order of convergence for the pressure is the same as 
 that in Table \mref{b-3}.
It indicates that the analysis in Theorem \mref{t-sup} is sharp.
Here we have an order-two superconvergence in $L^2$-norm too, for
   the velocity.
But this is not proved in this manuscript.

\begin{table}[htb]
  \caption{ The errors  $\b e_h=\b u-\b I_h \b u$ and
   $\epsilon_h=p-p_I$ for \meqref{e-solution2}.}
\lab{b-4} 
\begin{center}  \begin{tabular}{c|rr|rr|rr|r}  
\hline & $ |\b e_h|_{L^2}$ &$h^n$ &  $ |\b e_h|_{H^1}$ & $h^n$ &
    $ \|\epsilon_h\|_{L^2}$   &$h^n$ & \\ \hline
   & \multispan{5} \hfil \Qe divergence-free element, $k=3$
	 \hfil &&\#it  \\ \hline
 2& 0.001196745&   & 0.024927233&   & 0.1147364&3.8&  4\\
 3& 0.000045519&4.7& 0.001383336&4.2& 0.0069166&4.1&  4\\
 4& 0.000000937&5.6& 0.000051730&4.7& 0.0004383&4.0&  4\\
 5& 0.000000016&5.8& 0.000001826&4.9& 0.0000276&4.0&  4\\ 
 6& 0.000000000&5.9& 0.000000060&4.9& 0.0000017&4.0&  4\\
     \hline 
\end{tabular}\end{center} \end{table}

\def\comment#1{}
 
\comment{ --------- run on 06/12/11-10:57:55 ----------------------
  u, inf, l2, h1,p-inf pde= 1 dsol= 42 #Uz= 204 RB= 1 deg= 2 1
   G-Lab pt= 1
   u,  l2, h1 p-l2 #Uz
  re( 3 2 5) lv( 12, i )
 1& 0.000000&0.0& 0.000000&0.0& 8.000000&0.0&   1\\
 2& 0.570990&***& 3.531380&***& 3.778333&1.1&  29\\
 3& 0.244967&1.2& 3.028368&0.2& 1.393887&1.4&  65\\
 4& 0.074335&1.7& 1.797533&0.8& 0.474767&1.6& 136\\
 5& 0.019849&1.9& 0.946426&0.9& 0.144299&1.7& 297\\
 6& 0.005080&2.0& 0.481087&1.0& 0.040650&1.8& 330\\
 7& 0.001282&2.0& 0.241917&1.0& 0.010872&1.9& 204\\
}

\noindent  Hunan Key Laboratory for Computation and Simulation in Science
   and Engineering, Xiangtan University, China, 411105.
   huangyq@xtu.edu.cn.

\noindent 
Department of Mathematical Sciences, University
     of Delaware, DE 19716. szhang@udel.edu.  
\end{document}